\documentclass[11pt]{article}
\usepackage[numbers,sort]{natbib}
\usepackage[bookmarks=true,bookmarksnumbered=true,colorlinks=true,allcolors=black]{hyperref}

\usepackage{graphicx}
\usepackage{amsmath,amssymb,amsfonts,amsthm}
\usepackage{here}

\renewcommand*{\baselinestretch}{1.25}

\usepackage{geometry}
\usepackage{enumitem}

\newtheorem{theorem}{Theorem}[section]
\newtheorem{lemma}{Lemma}[section]

\theoremstyle{definition}
\newtheorem{definition}{Definition}[section]
\newtheorem*{rmk*}{Remark}
\newtheorem{rmk}{Remark}[section]
\newtheorem{example}{Example}[section]
\newtheorem{assumption}{Assumption}

\numberwithin{equation}{section}

\makeatletter
    \renewcommand*{\section}{\@startsection{section}{1}{\z@}%
    {10pt}{5pt}{\reset@font\normalsize\bfseries}}
\makeatother
    
\makeatletter
    \renewcommand*{\subsection}{\@startsection{subsection}{2}{\z@}%
    {5pt}{5pt}{\reset@font\normalsize\mdseries\itshape}}
\makeatother

\makeatletter
    \renewcommand*{\subsubsection}{\@startsection{subsubsection}{3}{\z@}%
    {5pt}{5pt}{\reset@font\normalsize\mdseries\itshape}}
\makeatother

\makeatletter
\def\@listi{\leftmargin\leftmargini
  \topsep=.5\baselineskip 
  \partopsep=0pt \parsep=0pt \itemsep=0pt}
\let\@listI\@listi
\@listi
\def\@listii{\leftmargin\leftmarginii
  \labelwidth\leftmarginii \advance\labelwidth-\labelsep
  \topsep=0pt \partopsep=0pt \parsep=0pt \itemsep=0pt}
\def\@listiii{\leftmargin\leftmarginiii
  \labelwidth\leftmarginiii \advance\labelwidth-\labelsep
  \topsep=0pt \partopsep=0pt \parsep=0pt \itemsep=0pt}
\def\@listiv{\leftmargin\leftmarginiv
  \labelwidth\leftmarginiv \advance\labelwidth-\labelsep
  \topsep=0pt \partopsep=0pt \parsep=0pt \itemsep=0pt}
\makeatother

\allowdisplaybreaks

\newcommand{\bs}[1]{\boldsymbol{#1}}

\begin{document}
\title{Oracle inequalities for sign constrained generalized linear models}
\date{\today}
\author{Yuta Koike
\thanks{Graduate School of Mathematical Sciences, The University of Tokyo, 3-8-1 Komaba, Meguro-ku, Tokyo 153-8914 Japan}
\thanks{Department of Business Administration, Graduate School of Social Sciences, Tokyo Metropolitan University, Marunouchi Eiraku Bldg. 18F, 1-4-1 Marunouchi, Chiyoda-ku, Tokyo 100-0005 Japan}
\thanks{The Institute of Statistical Mathematics, 10-3 Midori-cho, Tachikawa, Tokyo 190-8562, Japan}
\thanks{CREST, Japan Science and Technology Agency}
\and 
Yuta Tanoue\footnotemark[3]}

\maketitle

\begin{abstract}
High-dimensional data have recently been analyzed because of data collection technology evolution.
Although many methods have been developed to gain sparse recovery in the past two decades, most of these methods require selection of tuning parameters. As a consequence of this feature, results obtained with these methods heavily depend on the tuning. 
In this paper we study the theoretical properties of sign-constrained generalized linear models with convex loss function, which is one of the sparse regression methods without tuning parameters.
Recent studies on this topic have shown that, in the case of linear regression, sign-constrains alone could be as efficient as the oracle method if the design matrix enjoys a suitable assumption in addition to a traditional compatibility condition. We generalize this kind of result to a much more general model which encompasses the logistic and quantile regressions.          
We also perform some numerical experiments to confirm theoretical findings obtained in this paper.
\vspace{3mm}


\end{abstract}


\section{Introduction}

Because of data collection technology evolution in recent years, it attracts considerable attentions to analyze high-dimensional data characterized by a large number of explanatory variables as compared to the sample size in many areas such as microarray analysis, text mining, visual recognition and finance.
When one analyzes such high-dimensional data, the key structural property exploited is the sparsity or the near sparsity of the regressors. 
The sparsity assumption, that is a few among huge number of variables only are relevant to response variables, appears realistic for high-dimensional data analysis. This assumption helps us to find crucial variables or reduce noisy irrelevant variables for particular phenomena. 

To construct a sparse estimator with high-dimensional regression, regularization typically plays a crucial role. 
So far, many regularization-based sparse regression methods are proposed, e.g.~Lasso \cite{tibshirani1996regression}, group Lasso \cite{yuan2006model}, adaptive Lasso \cite{zou2006adaptive}, square-root Lasso \cite{belloni2011square}, SCAD \cite{fan2001variable}, MCP \cite{zhang2010nearly}, Bridge regression \cite{frank1993statistical} and so on. 
Among these methods, the Lasso is the most famous one. The Lasso estimator for linear regression is given by the following equation:
\begin{equation}
 \label{eq:lasso}
  \hat{\beta}_{\text{Lasso}} = \mathrm{arg}\min_{\beta \in \mathbb{R}^{p}}\left\{ \|y - X \beta\|^{2}_{2} + \lambda \sum_{j = 1}^{p} |\beta_{j}|\right\},
\end{equation}
where $y$ is the response variable, $X$ is the design matrix and $p$ is the number of explanatory variables.
By adding the penalty term $\lambda \sum_{j = 1}^{p} |\beta_{j}|$, where $\lambda$ is a tuning parameter, to ordinary least square regression, one can get a sparse estimator for $\beta$. 
Although the Lasso has some statistically good properties, these properties depend on the tuning parameter $\lambda$. A typical way of choosing the tuning parameter $\lambda$ is cross-validation, but it takes much computational cost.\footnote{To avoid such a computational burden, some authors have developed analytical methods for choosing the tuning parameter $\lambda$ such as using an information criterion; see \cite{NK2016,USMN2015} and references therein for details.} In addition, the estimation results may depend heavily on the choice of the tuning parameter $\lambda$.

Sign-constraints on the regression coefficients are a simpler constraint-based regression method. Basically, sign-constraints on the regression coefficients are studied in the context of linear least squares, which is referred to as non-negative least squares (NNLS) in the literature. 
We also remark that such a sign-constraint canonically appears in many fields; see e.g.~the Introduction of \cite{SH2013} and references therein. 
The NNLS estimator is obtained as a solution of the following optimization problem:
\begin{equation}
\label{eq:NNLS}
\mathrm{arg}\min_{\beta\in\mathbb{R}^p}\|y-X\beta\|^{2}_{2}\qquad\text{subject to }\min_{1\leq j\leq p}\beta_j\geq0.
\end{equation}
Unlike many regularization-based sparse regression methods, there is no tuning parameter in equation \eqref{eq:NNLS}. 
An early reference of NNLS is \citet{lawson1974solving}. \citet{lawson1974solving} define the NNLS problem as equation \eqref{eq:NNLS} and give a solution algorithm for NNLS.
Rather recently, NNLS in the noiseless setting like \cite{lawson1974solving} is studied by many authors such as  \cite{bruckstein2008uniqueness,donoho2010counting,wang2009conditions,wang2011unique}. 
These authors have studied the uniqueness of the solution of equation \eqref{eq:NNLS} under some conditions. 
What is important in our context is that, under some regularity conditions, they have shown that sign-constraints alone can make estimator sparse without additional regulation terms.
In the meantime, \citet{Meinshausen2013} and \citet{SH2013} have recently studied the NNLS estimator in the setting with Gaussian random errors from a statistical point of view.\footnote{We remark that the statistical property of the NNLS estimator in a fixed-dimensional setting has been extensively studied in the literature from a long ago; see e.g.~\citet{JT1966} and \citet{Liew1976}. More generally, the asymptotic properties of maximum likelihood estimators and M-estimators under general constraints have been developed by several authors such as \citet{LeCam1970} and \citet{Shapiro1989}; see also Section 4.9 of \citet{SS2005} and references therein.} An important conclusion of these studies is that sign-constraints themselves can be as effective as more explicit regularization methods such as the Lasso if the design matrix satisfies a special condition called the \textit{positive eigenvalue condition} (which is called the \textit{self-regularizing property} in \cite{SH2013}) in addition to a compatibility condition, where the latter condition is a standard one to prove the theoretical properties of regularization methods.    
More precisely, \citet{Meinshausen2013} has derived oracle inequalities for the $\ell_{1}$-error of the regression coefficient estimates and the prediction error under the afore-mentioned conditions. In particular, it has been shown that the NNLS estimator can achieve the same convergence rate as that of the Lasso.
\citet{SH2013} have also obtained similar results to those of \cite{Meinshausen2013}, but they have additionally derived oracle inequalities for the $\ell_{q}$-error of the regression coefficient estimates for $q\in[1,2]$ and $q=\infty$ under supplementary assumptions.  


The aim of this paper is to extend the results obtained in \cite{Meinshausen2013,SH2013} to a more general setup allowing general convex loss functions and non-linearity of response variables with respect to explanatory variables. This type of extension has been considered for regularization estimators; for the Lasso in \cite{vdG2007,vdG2008}, for the $\ell_p$-penalization with $p\geq1$ in \cite{Koltchinskii2009} and for the elastic net in \cite{CK2016}; see also Chapter 6 of \cite{BvdG2011}. In this paper we adopt an analogous setting to those articles and derive oracle inequalities for sign-constrained empirical loss minimization, which are parallel to those obtained in the above articles, with assuming the positive eigenvalue condition in addition to a compatibility condition. It would be worth mentioning that bounds of oracle inequalities for regularized estimation are explicitly connected with the regularization parameter (the tuning parameter $\lambda$ for the case of the Lasso), but this is not the case of sign-constrained estimation because it contains no regularization (or tuning) parameter. Interestingly, although our oracle inequalities for sign-constrained estimation contain a parameter controlling the tradeoff between the occurrence probabilities of the inequalities and the sharpness of the bounds of the inequalities as usual for regularized estimation (and this parameter usually corresponds to the regularization parameter), we do not need to specify this parameter when implementing the estimation; the estimation procedure automatically adjusts this parameter appropriately.     

The remainder of this paper is organized as follows. In Section \ref{sec:Setting} we explain a basic setting adopted in this paper. Section \ref{sec:Main result} presents main results obtained in this paper, while in Section \ref{sec:Example} we apply those main results to the case of Lipschitz loss functions, which contains logistic and quantile regressions, and derive concrete oracle inequalities. We complement the theoretical results by conducting numerical experiments in Section \ref{sec:Numerical simulation}. 

\subsection{Notation}

For a set $S$, we denote by $|S|$ the cardinality of $S$. 
For $\beta=(\beta_1,\dots,\beta_p)^\top\in\mathbb{R}^p$ and $S\subset\{1,\dots,p\}$, we define
\[
\beta_{j,S}=\beta_j1_{\{j\in S\}},\qquad j=1,\dots,p
\]
and set $\beta_S=(\beta_{1,S},\dots,\beta_{p,S})^\top$. For $\beta=(\beta_1,\dots,\beta_p)\in\mathbb{R}^p$ and $q>0$, we denote by $\|\beta\|_q$ the $\ell_q$-norm of $\beta$, i.e.
\[
\|\beta\|_q^q=\sum_{j=1}^p|\beta_j|^q.
\]
Also, we write $\beta\succeq0$ if $\beta_j\geq0$ for all $j=1,\dots,p$. 

\section{Setting}
\label{sec:Setting}
We basically follow Section 6.3 of \cite{BvdG2011} (see also \cite{vdG2007,CK2016}). We consider a probability space $(\Omega,\mathcal{F},P)$. Let $\mathbf{F}$ be a real inner product space and $L(f,\omega)$ be a real-valued function on $\mathbf{F}\times\Omega$. $\langle\cdot,\cdot\rangle$ denotes the inner product of $\mathbf{F}$ and $\|\cdot\|$ denotes the norm of $\mathbf{F}$. For every $f\in\mathbf{F}$, we assume that $L(f)=L(f,\cdot)\in L^1(P)$. We regard $\mathbf{F}$ as a (rich) parameter space and $L(f)$ as an empirical loss of $f\in\mathbf{F}$. In a regression setting, $\mathbf{F}$ typically consists of real-valued measurable functions on some measurable space $\mathcal{X}$, and $L(f)$ is typically of the form
\begin{equation}\label{loss}
L(f)=\frac{1}{n}\sum_{i=1}^n\ell(f(X_i),Y_i),
\end{equation}
where $X_1,\dots,X_n$ are random variables taking values in $\mathcal{X}$ (covariates), $Y_1,\dots,Y_n$ are random variables taking values in some Borel subset $\mathcal{Y}$ of $\mathbb{R}$ (response variables), and $\ell:\mathbb{R}\times\mathcal{Y}\to\mathbb{R}$ is a measurable function (loss function).  

We define the expected loss function $\bar{L}:\mathbf{F}\to\mathbb{R}$ by
\[
\bar{L}(f)=E[L(f)],\qquad f\in\mathbf{F}.
\]
Our target quantity is the minimizer $f^0\in\mathbf{F}$ of $\bar{L}(f)$ over $f\in\mathbf{F}$:
\[
\bar{L}(f^0)=\min_{f\in\mathbf{F}}\bar{L}(f)
\]
(we assume that the existence of such an element $f^0$). We then define the excess risk $\mathcal{E}:\mathbf{F}\to[0,\infty)$ by
\[
\mathcal{E}(f)=\bar{L}(f)-\bar{L}(f^0),\qquad f\in\mathbf{F}.
\]
We estimate the target $f^0$ by a linear combination of the elements $\psi_1,\dots,\psi_p$ of $\mathbf{F}$. Namely, our estimator is of the form
\[
f_\beta:=\sum_{j=1}^p\beta_j\psi_j
\]
for some $\beta=(\beta_1,\dots,\beta_p)^\top\in\mathbb{R}^p$. 
We estimate the regression coefficients by a non-negatively constrained minimizer $\hat{\beta}$ of the empirical loss $L(f_\beta)$. Namely, $\hat{\beta}$ is a solution of the following equation:
\[
L(f_{\hat{\beta}})=\min_{\beta\succeq0}L(f_\beta).
\]

\section{Main results}
\label{sec:Main result}

\subsection{Assumptions}

We start with introducing regularity conditions on the loss function $L$.
\begin{assumption}[Convexity]\label{convex}
The map $\beta\mapsto L(f_\beta,\omega)$ is convex for all $\omega\in\Omega$. 
\end{assumption}

\begin{assumption}[Quadratic margin]\label{margin}
There is a constant $c>0$ and a subset $\mathbf{F}_\text{local}$ of $\mathbf{F}$ such that
\[
\mathcal{E}(f)\geq c\|f-f^0\|^2
\]
for all $f\in\mathbf{F}_\text{local}$.
\end{assumption}
Assumptions \ref{convex}--\ref{margin} are more or less commonly employed in the literature and satisfied by many standard loss functions such as the logistic loss and check loss functions (see Section \ref{sec:Example}). 

Next we introduce regularity conditions on the Gram matrix $\Sigma$, which is defined by
\[
\Sigma=(\Sigma^{ij})_{1\leq i,j\leq p}=(\langle\psi_i,\psi_j\rangle)_{1\leq i,j\leq p}.
\]
The first condition is known as the \textit{compatibility condition}, which was originally introduced in \cite{vdG2007} and is commonly used to derive oracle inequalities for regularized estimators. 
\begin{definition}[Compatibility condition]
Let $S$ be a subset of $\{1,\dots,p\}$ and $L$ be a non-negative number. We say that the \textit{$(L,S)$-compatibility condition} is satisfied, with constant $\phi(L,S)>0$, if for all $\beta\in\mathbb{R}^p$ such that $\|\beta_S\|_1=1$ and $\|\beta_{S^c}\|_1\leq L$, it holds that
\[
|S|\beta^\top\Sigma\beta\geq\phi(L,S)^2.
\] 
\end{definition}
The compatibility condition is closely related to (and weaker than) the so-called \textit{restricted eigenvalue condition}, which is introduced in \cite{BRT2009} and \cite{Koltchinskii2009}. For two sets $S\subset\mathcal{N}\subset\{1,\dots,p\}$ and $L\geq0$, we define the following restricted set of $\beta$'s:
\[
\mathcal{R}(L,S,\mathcal{N}):=\left\{\beta\in\mathcal{R}^p:\|\beta_{S^c}\|_1\leq L\|\beta_S\|_1,\|\beta_{\mathcal{N}^c}\|_\infty\leq\min_{j\in\mathcal{N}\setminus S}|\beta_j|\right\}.
\]
\begin{definition}[Restricted eigenvalue condition]
Let $S$ be a subset of $\{1,\dots,p\}$, $L$ be a non-negative number, and $N$ be an integer satisfying $|S|\leq N\leq p$. We say that the \textit{$(L,S,N)$-restricted eigenvalue condition} is satisfied, with constant $\phi(L,S,N)>0$, if for all set $\mathcal{N}\subset\{1,\dots,p\}$ with $\mathcal{N}\supset S$ and $|\mathcal{N}|=N$, and all $\beta\in\mathcal{R}(L,S,\mathcal{N})$, it holds that
\[
\|\beta_{\mathcal{N}}\|_2^2\leq\beta^\top\Sigma\beta/\phi(L,S,N)^2.
\]
\end{definition}
Note that, if the $(L,S,N)$-restricted eigenvalue condition is satisfied with constant $\phi(L,S,N)>0$, the $(L,S)$-compatibility condition is satisfied with the same constant $\phi(L,S,N)$. We refer to \cite{vdGB2009} and Section 6.13 of \cite{BvdG2011} for discussions about these conditions and related ones appearing in the literature. 

The next condition, which is independently introduced in \citet{Meinshausen2013} and \citet{SH2013} respectively, plays a crucial role in derivation of oracle inequalities for non-negatively constrained estimation.  
\begin{definition}[Positive eigenvalue condition\footnote{In \cite{SH2013} this condition is called the \textit{self-regularizing property}.}]
We say that the \textit{positive eigenvalue condition} is satisfied, with constant $\tau>0$, if for all $\beta\in\mathbb{R}^p$ such that $\|\beta_S\|_1=1$ and $\beta\succeq0$, it holds that
 \begin{equation}
  \label{eq:positive}
 \beta^\top\Sigma\beta\geq\tau^2.
 \end{equation}

\end{definition}
There are several relevant situations where the compatibility/restricted eigenvalue condition and the positive eigenvalue condition are simultaneously satisfied with proper constants: See Section 2.2 of \cite{Meinshausen2013} and Section 5 of \cite{SH2013} for details. 

\subsection{Oracle inequalities}
Throughout this section, we fix a vector $\beta^*\in\mathbb{R}^p$ such that $\beta^*\succeq0$. Typically, $\beta^*$ is the non-negatively constrained minimizer of $\bar{L}(f_\beta)$ with respect to $\beta$:
\[
\bar{L}(\beta^*)=\min_{\beta\succeq0}\bar{L}(f_\beta),
\]
but other choices are possible. For each $\beta\in\mathbb{R}^p$, we set 
\[
v(\beta)=L(f_\beta)-\bar{L}(f_\beta).
\]
For $M>0$, we set
\[
\mathbf{Z}_M=\sup_{\beta:\|\beta-\beta^*\|_1\leq M}|v(\beta)-v(\beta^*)|
\]

Set $f^*=f_{\beta^*}$, $\hat{f}=f_{\hat{\beta}}$, $S_*=\{j:\beta_j^*\neq0\}$, $s_*=|S_*|$ and $C:=\max_{1\leq j\leq p}\|\psi_j\|^2$. We first derive oracle inequalities for the $\ell_1$-error of coefficient estimates and the prediction error of the non-negatively constrained estimation. 
\begin{theorem}\label{mainthm}
Suppose that Assumptions \ref{convex}--\ref{margin} are fulfilled. Suppose also that the positive eigenvalue condition is satisfied with constant $\tau>0$ and the $(3C/\tau^2,S_*)$-compatibility condition is satisfied with constant $\phi_*^2>0$. Let $\lambda>0$ be a positive number and set
\begin{equation}\label{lambda}
\varepsilon_\lambda=4\mathcal{E}(f^*)+\frac{12}{c\phi_*^2}\left(1+\frac{3C}{\tau^2}\right)^2s_*\lambda^2,\qquad
M_\lambda=\varepsilon_\lambda/\lambda.
\end{equation}
Suppose further that $f_\beta\in\mathbf{F}_\mathrm{local}$ for any $\beta\in\mathbb{R}^p$ whenever $\|\beta-\beta^*\|_1\leq M_\lambda$. Then, on the set $\{\mathbf{Z}_{M_\lambda}\leq\varepsilon_\lambda\}$ we have
\begin{equation}\label{theorem:l1}
\|\hat{\beta}-\beta^*\|_1\leq M_\lambda=\frac{4}{\lambda}\mathcal{E}(f^*)+\frac{12}{c\phi_*^2}\left(1+\frac{3C}{\tau^2}\right)^2s_*\lambda
 \end{equation}
and
\begin{equation}\label{theorem:prediction}
\mathcal{E}(\hat{f})\leq\frac{5}{4}\varepsilon_\lambda=5\mathcal{E}(f^*)+\frac{15}{c\phi_*^2}\left(1+\frac{3C}{\tau^2}\right)^2s_*\lambda^2. 
\end{equation}
\end{theorem}

\begin{rmk}
Following the spirit of \cite{vdG2007} and Chapter 6 of \cite{BvdG2011}, we consider the deterministic problem separately from the stochastic one in Theorem \ref{mainthm}. Namely, we state inequalities \eqref{theorem:l1}--\eqref{theorem:prediction} on the event $\{\mathbf{Z}_{M_\lambda}\leq\varepsilon_\lambda\}$ and do not refer to the occurrence probability of the event $\{\mathbf{Z}_{M_\lambda}\leq\varepsilon_\lambda\}$. The advantage of this approach is that we can apply our result to various situations (serially dependent data for example) once we have an appropriate probability inequality for the event $\{\mathbf{Z}_{M_\lambda}\leq\varepsilon_\lambda\}$. 
\end{rmk}

\begin{rmk}
Since inequalities \eqref{theorem:l1}--\eqref{theorem:prediction} are valid only on the set $\{\mathbf{Z}_{M_\lambda}\leq\varepsilon_\lambda\}$, we should choose $\lambda$ such that the set $\{\mathbf{Z}_{M_\lambda}\leq\varepsilon_\lambda\}$ has large probability. This is typically achieved by taking a large value of $\lambda$. In the meantime, inequalities \eqref{theorem:l1}--\eqref{theorem:prediction} are sharper as $\lambda$ is smaller, so there is a tradeoff between the sharpness of the inequalities and the probability with which the inequalities hold true. In typical situations $\lambda$ can be taken so that $\lambda$ is of order $\sqrt{\log p/n}$, where $n$ denotes the sample size (see Examples below as well as Sections 14.8--14.9 of \cite{BvdG2011}). Therefore, as long as the constants $\phi_*^2$ and $\tau^2$ are of order 1, our non-negatively constrained empirical loss minimization procedure attains the same order as that of the Lasso.
\end{rmk}

\begin{rmk}
The compatibility constant $\phi_*$ and the positive eigenvalue constant $\tau$ enter the bounds obtained in Theorem \ref{mainthm} with the form $\phi_*^{-2}+(\phi_*\tau)^{-2}$, which is similar to Theorem 2 of \cite{SH2013}. 
\end{rmk}

\begin{proof}[\bfseries\upshape Proof of Theorem \ref{mainthm}]
The basic strategy of the proof is the same as that of Theorem 6.4 from \cite{BvdG2011}. Throughout the proof, we assume that we are on the set $\{\mathbf{Z}_{M_\lambda}\leq\varepsilon_\lambda\}$. Set
\[
t=\frac{M_\lambda}{M_\lambda+\|\hat{\beta}-\beta^*\|_1},\qquad
\tilde{\beta}=t\hat{\beta}+(1-t)\beta^*,\qquad
\tilde{f}=f_{\tilde{\beta}}.
\]
Noting that $\|\tilde{\beta}-\beta^*\|_1\leq M_\lambda$, we have
\begin{align*}
\mathcal{E}(\tilde{f})
&=-\{v(\tilde{\beta})-v(\beta^*)\}+L(\tilde{f})-\bar{L}(f^0)-v(\beta^*)
\leq\mathbf{Z}_{M_\lambda}+L(\tilde{f})-\bar{L}(f^0)-v(\beta^*).
\end{align*}
Using Assumption \ref{convex}, we obtain
\begin{align*}
\mathcal{E}(\tilde{f})
\leq\mathbf{Z}_{M_\lambda}+tL(\hat{f})+(1-t)L(f^*)-\bar{L}(f^0)-v(\beta^*).
\end{align*}
Now, by the definition of $\hat{\beta}$ we have $L(\hat{f})\leq L(f^*)$, hence we have
\begin{align*}
\mathcal{E}(\tilde{f})
\leq\mathbf{Z}_{M_\lambda}+L(f^*)-\bar{L}(f^0)-v(\beta^*)
=\mathbf{Z}_{M_\lambda}+\mathcal{E}(f^*).
\end{align*}
Next set $\delta=\tilde{\beta}-\beta^*$. Then we have
\[
\|f_\delta\|^2=\|(\tilde{f}-f^0)-(f^*-f^0)\|^2\leq 2(\|\tilde{f}-f^0\|^2+\|f^*-f^0\|^2).
\]
Since $\tilde{f},f^*\in\mathbf{F}_\text{local}$ by assumption, we obtain
\begin{equation}\label{basic}
c\|f_\delta\|^2\leq2\{\mathcal{E}(\tilde{f})+\mathcal{E}(f^*)\}\leq2\{\mathbf{Z}_{M_\lambda}+2\mathcal{E}(f^*)\}\leq3\varepsilon_\lambda
\end{equation}
by Assumption \ref{margin}. Now if $\|\delta_{S_*^c}\|_1\leq(3C/\tau^2)\|\delta_{S_*}\|_1$, by the definition of $\phi_*^2$ we have
\begin{align*}
\|\delta_{S_*}\|_1^2\leq\frac{s_*\delta^\top\Sigma\delta}{\phi_*^2}=\frac{s_*\|f_\delta\|^2}{\phi_*^2}.
\end{align*}
Therefore, by \eqref{basic} we have
\begin{align*}
\|\delta_{S_*}\|_1\leq\sqrt{\frac{s_*\cdot3\varepsilon_\lambda}{c\phi_*^2}}.
\end{align*}
Hence it holds that
\begin{align*}
\|\delta\|_1=\|\delta_{S_*}\|_1+\|\delta_{S_*^c}\|_1\leq\left(1+\frac{3C}{\tau^2}\right)\sqrt{\frac{s_*\cdot3\varepsilon_\lambda}{c\phi_*^2}}\leq\frac{M_\lambda}{2}.
\end{align*}
On the other hand, if $\|\delta_{S_*^c}\|_1>(3C/\tau^2)\|\delta_{S_*}\|_1$, we have 
\begin{align*}
\|f_\delta\|^2&=\delta^\top\Sigma\delta=\delta_{S^c_*}^\top\Sigma\delta_{S^c_*}+2\delta_{S^c_*}^\top\Sigma\delta_{S_*}+\delta_{S_*}^\top\Sigma\delta_{S_*}\\
&\geq\delta_{S^c_*}^\top\Sigma\delta_{S^c_*}-2C\|\delta_{S^c_*}\|_1\|\delta_{S_*}\|_1
\geq\delta_{S^c_*}^\top\Sigma\delta_{S^c_*}-\frac{2\tau^2}{3}\|\delta_{S^c_*}\|_1^2.
\end{align*}
Since $\delta_{S_*^c}\succeq0$ by definition, we obtain
\begin{align*}
\|f_\delta\|^2\geq\frac{\tau^2}{3}\|\delta_{S^c_*}\|_1^2
\end{align*}
by the definition of $\tau^2$. Hence \eqref{basic} implies that
\begin{align*}
\|\delta_{S^c_*}\|_1^2\leq\frac{9\varepsilon_\lambda}{c\tau^2},
\end{align*}
and thus we obtain
\begin{align*}
\|\delta\|_1=\|\delta_{S_*}\|_1+\|\delta_{S_*^c}\|_1\leq
\left(1+\frac{\tau^2}{3C}\right)\|\delta_{S^c_*}\|_1\leq\left(3+\frac{\tau^2}{C}\right)\sqrt{\frac{\varepsilon_\lambda}{c\tau^2}}.
\end{align*}
Now, by definition we have $\tau^2\leq C$ and $\phi_*^2\leq s_*C$, hence it holds that
\begin{align*}
\|\delta\|_1\leq4\sqrt{\frac{Cs_*}{c\phi_*^2\tau^2}\lambda M_\lambda}\leq\frac{M_\lambda}{2}.
\end{align*}
Consequently, in either cases we obtain $\|\delta\|_1\leq M_\lambda/2$. Therefore, we have
\begin{align*}
\|\hat{\beta}-\beta^*\|_1=t^{-1}\|\delta\|_1
\leq(M_\lambda+\|\hat{\beta}-\beta^*\|_1)/2,
\end{align*}
hence $\|\hat{\beta}-\beta^*\|_1\leq M_\lambda$. This means \eqref{theorem:l1}. Moreover, we have
\begin{align*}
\mathcal{E}(\hat{f})=-\{v(\tilde{\beta})-v(\beta^*)\}+L(\hat{f})-\bar{L}(f^0)-v(\beta^*)
\leq\mathbf{Z}_{M_\lambda}+L(\hat{f})-\bar{L}(f^0)-v(\beta^*).
\end{align*}
Now, by the definition of $\hat{\beta}$ it holds that $L(\hat{f})\leq L(f^*)$, hence we conclude that
\begin{align*}
\mathcal{E}(\hat{f})
\leq\mathbf{Z}_{M_\lambda}+L(f^*)-\bar{L}(f^0)-v(\beta^*)
=\mathbf{Z}_{M_\lambda}+\mathcal{E}(f^*)\leq\frac{5}{4}\varepsilon_\lambda.
\end{align*}
This completes the proof of \eqref{theorem:prediction}.
\end{proof}

Next we derive oracle inequalities for the $\ell_q$-error of coefficient estimates for all $q\in[1,2]$. We only focus on the case of $f^0=f^*$ for simplicity.  
\begin{theorem}\label{lq-oracle}
Suppose that Assumptions \ref{convex}--\ref{margin} are fulfilled. Suppose also that the positive eigenvalue condition is satisfied with constant $\tau>0$ and the $(3C/\tau^2,S_*,2s_*)$-restricted eigenvalue condition is satisfied with constant $\phi_*\in(0,1]$. Let $\lambda>0$ and define $\varepsilon_\lambda$ and $M_\lambda$ as in \eqref{lambda}. 
Suppose further that $f^0=f^*$ and $f_\beta\in\mathbf{F}_\mathrm{local}$ for any $\beta\in\mathbb{R}^p$ whenever $\|\beta-\beta^*\|_1\leq M_\lambda$. Then, on the set $\{\mathbf{Z}_{M_\lambda}\leq\varepsilon_\lambda\}$ we have
\begin{equation}\label{theorem:lq}
\|\hat{\beta}-\beta^*\|_q^q\leq\left\{\frac{12}{c\phi_*^2}\left(1+\frac{3C}{\tau^2}\right)^2\right\}^qs_*\lambda^q
\end{equation}
for any $q\in[1,2]$. 
\end{theorem}

\begin{rmk}
To derive oracle inequalities for the $\ell_q$-error of coefficient estimates for $q\in[1,2]$, we need to impose a restricted eigenvalue condition stronger than the compatibility condition imposed in Theorem \ref{mainthm}, which is standard in the literature of regularized estimation; see e.g.~Section 6.8 of \cite{BvdG2011}. 
\end{rmk}

\begin{rmk}
The compatibility constant $\phi_*$ and the positive eigenvalue constant $\tau$ enter the bounds obtained in Theorem \ref{lq-oracle} with the form $\{\phi_*^{-2}+(\phi_*\tau)^{-2}\}^q$, which is again similar to Theorem 2 of \cite{SH2013}. 
\end{rmk}

\begin{proof}[\bfseries\upshape Proof of Theorem \ref{lq-oracle}]
Set
\[
\kappa=\frac{12}{c\phi_*^2}\left(1+\frac{3C}{\tau^2}\right)^2.
\]
Noting that $\mathcal{E}(f^*)=0$ because $f^0=f^*$, we have $\varepsilon_\lambda=\kappa s_*\lambda$. 

Let us sort the elements $j_1,\dots,j_{|S_*^c|}$ of the set $S_*^c$ as $|\hat{\beta}_{j_1}|\geq\cdots\geq|\hat{\beta}_{j_{|S_*^c|}}|$, and set $\mathcal{N}=S_*\cup\{j_1,\dots,j_{s_*}\}$. Then, by Lemma 6.9 of \cite{BvdG2011} we have
\begin{align*}
\|\hat{\beta}_{\mathcal{N}^c}\|_q\leq s_*^{-(q-1)/q}\|\hat{\beta}_{S_*^c}\|_1.
\end{align*}
Next, noting that $f^0=f^*$ and $\phi(3C/\tau^2,S_*,2s_*)\leq\phi(3C/\tau^2,S_*)$, an analogous argument to the proof of Theorem \ref{mainthm} yields
\[
\mathcal{E}(\hat{f})\leq\mathbf{Z}_{M_\lambda}\leq\varepsilon_\lambda.
\] 
Now let us set $\delta=\hat{\beta}-\beta^*$. If $\|\delta_{S_*^c}\|_1\leq(3C/\tau^2)\|\delta_{S_*}\|_1$, the $(3C/\tau^2,S_*,2s_*)$-restricted eigenvalue condition implies that
\[
\|\delta_{\mathcal{N}}\|_2\leq\|f_\delta\|/\phi_*,
\]
hence the quadratic margin condition yields
\[
\|\delta_{\mathcal{N}}\|_2^2\leq \mathcal{E}(\hat{f})/c\phi_*^2\leq\varepsilon_\lambda/c\phi_*^2.
\]
On the other hand, the inequality $\|\delta_{S_*^c}\|_1\leq(3C/\tau^2)\|\delta_{S_*}\|_1$ and the Schwarz inequality imply that
\[
\|\hat{\beta}_{S_*^c}\|_1=\|\delta_{S_*^c}\|_1\leq(3C/\tau^2)\|\delta_{S_*}\|_1\leq(3C/\tau^2)\sqrt{s_*}\|\delta_{\mathcal{N}}\|_2.
\]
Consequently, using the H\"older inequality, we obtain
\begin{align*}
\|\delta\|_q^q&=\|\delta_{\mathcal{N}}\|_q^q+\|\delta_{\mathcal{N}^c}\|_q^q
\leq(2s_*)^{1-q/2}\|\delta_{\mathcal{N}}\|_2^q+\|\hat{\beta}_{\mathcal{N}^c}\|_q^q
\leq\left\{2^{1-q/2}+(3C/\tau^2)^q\right\}s_*^{1-q/2}\|\delta_{\mathcal{N}}\|_2^q\\
&\leq\left\{2^{1-q/2}+(3C/\tau^2)^q\right\}s_*^{1-q/2}\varepsilon_\lambda^{q/2}/c^{q/2}\phi_*^q.
\end{align*}
Since $q\geq1$, we conclude that
\[
\|\delta\|_q^q\leq\kappa^{q/2}s_*^{1-q/2}\kappa^{q/2}s_*^{q/2}\lambda^{q/2}=\kappa^qs_*\lambda^{q/2}.
\]
In the meantime, if $\|\delta_{S_*^c}\|_1>(3C/\tau^2)\|\delta_{S_*}\|_1$, the same argument as in the proof of Theorem \ref{mainthm} yields
\[
\|\delta\|_1\leq\left(3+\frac{\tau^2}{C}\right)\sqrt{\frac{\varepsilon_\lambda}{c\tau^2}}.
\]
Therefore, noting that $\phi(3C/\tau^2,S_*,2s_*)\leq1$ and $\tau^2\leq C$, we obtain
\begin{align*}
\|\delta\|_q^q\leq\|\delta\|_1^q
\leq4^q\left\{\frac{C}{c\phi_*^2\tau^2}\varepsilon_\lambda\right\}^{q/2}
\leq\left\{\frac{4^{q-1}}{3^q}+\left(\frac{3C}{\tau^2}\right)^q\right\}\frac{\varepsilon_\lambda^{q/2}}{c^{q/2}\phi_*^q}.
\end{align*}
Since $q\leq2$, we conclude that
\[
\|\delta\|_q^q\leq\kappa^{q/2}\cdot \kappa^{q/2}s_*^{q/2}\lambda^{q/2}
\leq \kappa^qs^*\lambda^{q/2}.
\]
Hence we complete the proof.
\end{proof}

\section{Application to Lipschitz loss functions}
\label{sec:Example}


The main results presented in Section \ref{sec:Main result} are general but rather abstract, hence their implications are less understandable. Therefore, in this section we apply our results to the case that the loss function is Lipschitz continuous and derive more concrete versions of oracle inequalities for logistic and quadratic regressions. 
  
The setting is basically adopted from Section 6 of \cite{vdGM2012}. Let us consider the case that the empirical loss function $L$ is given by \eqref{loss}. We assume that $X_1,\dots,X_n$ are fixed covariates, i.e.~they are deterministic, and that $Y_1,\dots,Y_n$ are independent. Let $Q_n$ be the empirical probability measure generated by $\{X_i\}_{i=1}^n$, i.e.
\[
Q_n(A)=\frac{1}{n}\sum_{i=1}^n1_A(X_i)
\]
for each measurable set $A$ of $\mathcal{X}$. Then, we take $\mathbf{F}$ as a subspace of $L^2(Q_n)$. We assume $f^0=f_{\beta^*}$ for some $\beta^*\in\mathbb{R}^p$ such that $\beta^*\succeq0$. We set $K_X:=\max_{1\leq j\leq p}\max_{1\leq i\leq n}|\psi_j(X_i)|$ and $K_0:=\max_{1\leq i\leq n}|f^0(X_i)|$. 
 
We impose the following conditions on the loss function: 
\begin{enumerate}[label={\normalfont[A\arabic*]}]

\item\label{hypo:convex} The map $a\mapsto\ell(a,y)$ is convex for all $y\in\mathcal{Y}$.

\item\label{hypo:Lipschitz} There is a constant $c_L>0$ such that
\begin{equation*}
|\ell(a,y)-\ell(b,y)|\leq c_L|a-b|
\end{equation*} 
for all $a,b\in\mathbb{R}$ and $y\in\mathcal{Y}$. 

\item\label{hypo:non-d} For every $i=1,\dots,n$, the function $\bar{\ell}_i:\mathbb{R}\to\mathbb{R}$ defined by $\bar{\ell}_i(a)=E[\ell(a,Y_i)]$ for $a\in\mathbb{R}$ is twice differentiable and satisfies $\inf_{|a|\leq K_X+K_0}\bar{\ell}_i''(a)>0$.

\end{enumerate}
\begin{theorem}\label{theorem:Lipschitz}
Suppose that \ref{hypo:convex}--\ref{hypo:non-d} are satisfied. Suppose also that the positive eigenvalue condition is satisfied with constant $\tau>0$ and the $(3C/\tau^2,S_*,2s_*)$-restricted eigenvalue condition is satisfied with constant $\phi_*\in(0,1]$. Set
\[
\lambda=2c_LC\left(4\sqrt{\frac{2\log(2p)}{n}}+\sqrt{\frac{2t}{n}}\right)
\]
for some $t>0$ and suppose that $M_\lambda\leq1$. Then, with probability larger than $1-e^{-t}$ we have \eqref{theorem:l1}--\eqref{theorem:prediction} and \eqref{theorem:lq} with $c:=\frac{1}{2}\min_{1\leq i\leq n}\inf_{|a|\leq K_X+K_0}\bar{\ell}_i''(a)$.
\end{theorem}

\begin{proof}
We apply Theorems \ref{mainthm}--\ref{lq-oracle}. Assumption \ref{convex} is satisfied by \ref{hypo:convex}. To check Assumption \ref{margin}, we set $\mathbf{F}_\text{local}:=\{f\in\mathbf{F}:\max_{1\leq i\leq n}|f(X_i)-f^0(X_i)|\leq K_X\}$. Since $M_\lambda\leq1$ and $f^0=f_{\beta^*}$, we have $f_\beta\in\mathbf{F}_\text{local}$ as long as $\|\beta-\beta^*\|_1\leq M_\lambda$. Under \ref{hypo:non-d}, by Taylor's theorem we have
\[
\bar{\ell}_i(f(X_i))-\bar{\ell}_i(f^0(X_i))\geq c(f(X_i)-f^0(X_i))^2
\]
for all $f\in\mathbf{F}_\text{local}$. Hence Assumption \ref{margin} is satisfied. Therefore, the proof is completed once we show $P(\mathbf{Z}_{M_\lambda}\leq\varepsilon_\lambda)\geq1-e^{-t}$. This follows from Example 14.2 of \cite{BvdG2011} due to \ref{hypo:Lipschitz}.
\end{proof}

\begin{example}[Logistic regression]
Let us consider the case $\mathcal{Y}=\{0,1\}$ and we take
\[
\ell(a,y)=-\{y\log G(a)+(1-y)\log(1-G(a))\}
\]
as the loss function, where
\[
G(a)=\frac{1}{1+e^{-a}}.
\]
Since $\frac{\partial\ell}{\partial a}(a,y)=-\{y-G(a)\}$ and $\frac{\partial^2\ell}{\partial a^2}(a,y)=G(a)(1-G(a))$, 
all the assumptions \ref{hypo:convex}--\ref{hypo:non-d} are satisfied. Therefore, we can derive an oracle inequality for the non-negatively constrained logistic regression from Theorem \ref{theorem:Lipschitz}.
\end{example}

\begin{example}[Quantile regression]
If we model the $\gamma$-quantile ($\gamma\in(0,1)$) of the distribution of $Y_i$ by $f(X_i)$, we take
\[
\ell(a,y)=\rho_\gamma(y-a)
\]
as the loss function, where
\[
\rho_\gamma(z)=\gamma|z|1_{\{z>0\}}+(1-\gamma)|z|1_{\{z\leq0\}}.
\]
In this case \ref{hypo:convex} is evidently true and \ref{hypo:Lipschitz} holds true with $c_L=1$. Moreover, if $Y_i$ has a density $g_i$ with $\inf_{|a|\leq K_X+K_0}g_i(a)>0$, \ref{hypo:non-d} holds true. Therefore, in this situation we can derive an oracle inequality for the non-negatively constrained quantile regression from Theorem \ref{theorem:Lipschitz}.
\end{example}

\section{Numerical simulation}
\label{sec:Numerical simulation}
In this section we conduct some numerical simulation to complement our theoretical findings obtained in the preceding sections. Specifically, we examine the logistic and least absolute deviation regressions in a setting analogous to the one of \cite{Meinshausen2013}. Following \cite{Meinshausen2013}, we compare the performance of sign-constrained regression to its Lasso counterpart with a cross-validated choice of the regularization parameter $\lambda$, in terms of their $\ell_1$-errors of coefficient estimation.\footnote{The Lasso version of logistic regression is implemented by the \textsf{R} package \textbf{glmnet}, while the Lasso version of logistic regression is implemented by the \textsf{R} package \textbf{rqPen}.}

\subsection{Toeplitz design}
\label{sec:Toeplitz}

The theoretical results shown in the preceding sections suggests that the positive eigenvalue condition should play a key role in success of sign-constrained regression. In fact, the inequalities of the theorems are shaper as the constant $\phi_*$ of the positive eigenvalue condition is larger. To confirm this point in the current experiment, we adopt the Toeplitz design to simulate the covariates. Specifically, following \cite{Meinshausen2013}, we draw covariates from the $p$-dimensional normal distribution $N_p(0, \Sigma)$, where the $(k,k')$-th component of the covariance matrix $\Sigma$ is given by
\begin{equation}
 \label{eq:Toeplitz}
\Sigma_{kk'} = \rho^{|k - k'|/p}    
\end{equation}
for $k,k' \in \{1, \cdots, p\}$ with some $\rho\in(0,1)$. In this case the positive eigenvalue condition is satisfied with constant $\rho^2$ (see Example I from page 1611 of \cite{Meinshausen2013} or Section 5.2 of \cite{SH2013}). Consequently, we may expect that the performance of sign-constrained regression would be more competitive with the Lasso as the parameter $\rho$ increases. We vary it as $\rho\in\{0.1,0.2,\dots,0.8,0.9\}$.  


\subsection{Logistic regression}
\subsubsection{Setting}

We begin by considering a logistic regression model:
\begin{equation}
 \label{eq:LogisticReg}
P(y=1|\bs{x}) = \frac{1}{1 + \exp(-\bs{x}^\top\beta)},
\end{equation}
where $y\in\{0,1\}$ is the response variable and $\bs{x}\sim N_p(0,\Sigma)$ is the vector of the covariates. 
The components of $\beta$ are identically zero, except for $s$ randomly chosen components that are all set to 1. We vary $s\in\{3,10,20\}$ to assess the effect of the sparsity of the model. We also vary the dimension as $p\in\{20,50,100,200,500,1000\}$.    
We simulate $n$ independent observations $\{(\bs{x}_i,y_i)\}_{i=1}^n$ for the pair $(\bs{x},y)$ from model \eqref{eq:LogisticReg}. We set $n=100$ in this experiment.   

\subsubsection{Results}
The results are shown in Figure \ref{fig:logit}. We use the logarithm of the ratio $\|\hat{\beta} - \beta^{*}\|_1/\|\hat{\beta}^{\lambda} - \beta^{*}\|_1 $ to evaluate the regression performance, where $\hat{\beta}$ and $\hat{\beta}^{\lambda}$ are the estimates of the regression coefficient $\beta$ by sign-constrained regression and the Lasso with a cross-validated choice of $\lambda$, respectively. 

We find from the figure that the relative estimation error of the sign-constrained regression is smaller as $\rho$ increases, which is in line with the theory developed in this paper (see Section \ref{sec:Toeplitz}). 
In particular, when the value of $\rho$ is close to 1, the sign-constrained regression tends to outperform the Lasso in the sparsest scenario $s=3$ in this experiment. 
This result confirms that sign-constrained regression performs well under strongly correlated design, which was also pointed out in \cite{Meinshausen2013}.

\begin{figure}[H]
 \centering
  \includegraphics[scale=0.6]{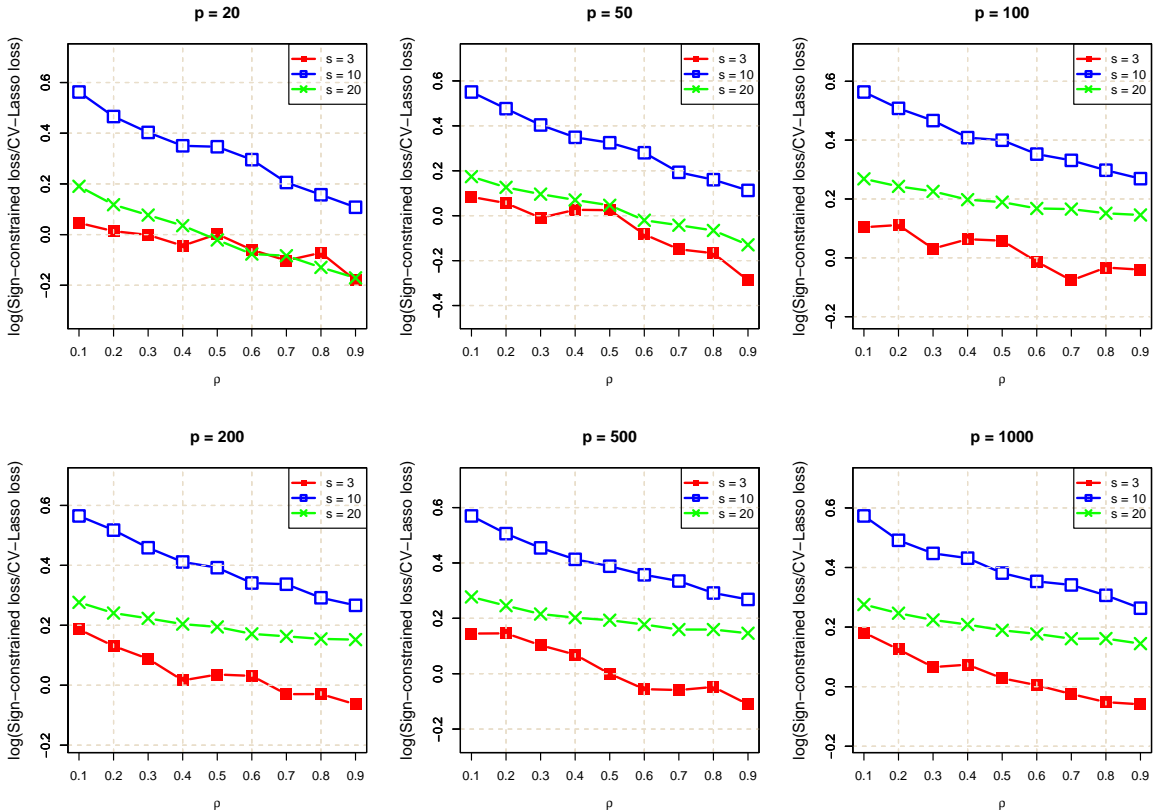}
 \caption{The logarithm of the ratio $\|\hat{\beta} - \beta^{*}\|_1/\|\hat{\beta}^{\lambda} - \beta^{*}\|_1 $ of the logistic regression (the average of 500 simulation runs).}
 \label{fig:logit}
\end{figure}

\subsection{Least Absolute Deviation regression}

\subsubsection{Setting}

Next we examine a linear regression model with Laplacian errors
\begin{equation}
 \label{eq:LADReg}
y = \bs{x}^\top\beta + \epsilon,
\end{equation}
where
$\bs{x} ~ \sim N_p(0, \Sigma)$ and $\epsilon$ is a random variable following the standard Laplace distribution.
As in the previous subsection, the components of $\beta$ are identically zero, except for $s$ randomly chosen components that are all set to 1 with $s\in\{3,10,20\}$. We simulate $n$ independent observations $\{(\bs{x}_i,y_i)\}_{i=1}^n$ for the pair $(\bs{x},y)$ from model \eqref{eq:LADReg}, and estimate the regression coefficient $\beta$ by least absolute deviations with sign-constraints or a Lasso-type penalty with a cross-validated choice of the regularization parameter $\lambda$. Since least absolute deviations is computationally more expensive than logistic regression due to the lack of differentiability of the loss function, we set $n=40$ and only consider the dimension $p \in \{20, 40,80\}$ to reduce computational time. 

\subsubsection{Results}
The results are shown in Figure \ref{fig:lad}. Similarly to the case of logistic regression, we use the logarithm of the ratio $\|\hat{\beta} - \beta^{*}\|_1/\|\hat{\beta}^{\lambda} - \beta^{*}\|_1 $ to evaluate the regression performance, where $\hat{\beta}$ and $\hat{\beta}^{\lambda}$ are the estimates of the regression coefficient $\beta$ by sign-constraint regression and the Lasso with a cross-validated choice of $\lambda$, respectively. 

Since the case where $p = 20$ with $s = 20$ is not a sparse setting, the corresponding result displays a different pattern from other settings. 
Except for this case, the relative estimation error of the sign-constrained regression tends to be smaller as $\rho$ increases. 
This is the same as in the case of logistic regression.
Furthermore, in the least absolute deviations regression case, sign-constrained regression's performance is better than the one of the Lasso in almost all simulation settings.


\begin{figure}[H]
 \centering
  \includegraphics[scale=0.6]{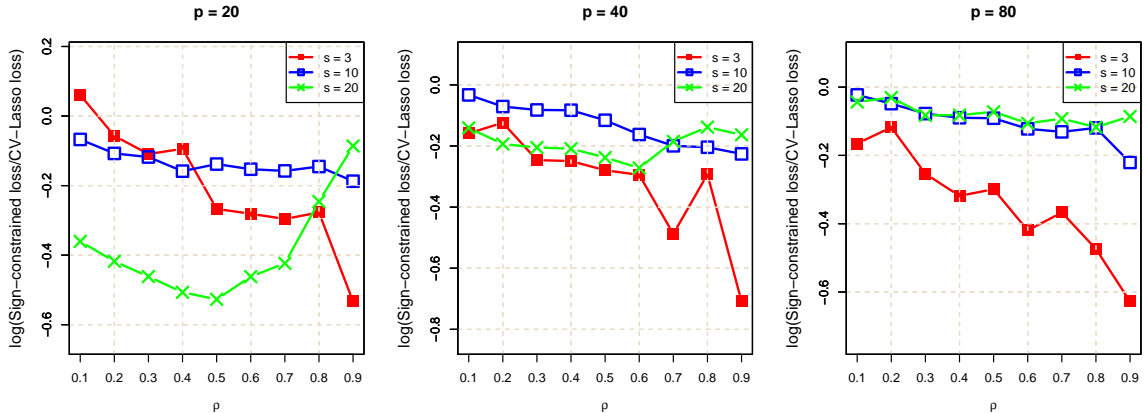}
 \caption{The logarithm of the ratio $\|\hat{\beta} - \beta^{*}\|_1/\|\hat{\beta}^{\lambda} - \beta^{*}\|_1 $ of the least absolute deviations regression (the average of 500 simulation runs).}
 \label{fig:lad}
\end{figure}

\section{Conclusion}

In this paper, 
we have studied the theoretical properties of sign-constrained regression with convex loss function under an analogous setting to the Lasso in \cite{vdG2007,vdG2008}, the $\ell_p$-penalization with $p\geq1$ in \cite{Koltchinskii2009} and the elastic net in \cite{CK2016}. 
As a result, we have derived oracle inequalities for the prediction error and the $\ell_q$-error of the coefficient estimates for $q\in[1,2]$ under the so-called positive eigenvalue condition in addition to a traditional compatibility condition. 
We also apply those results to the logistic regression and the quantile regression.

Furthermore, we have performed a numerical simulation analysis (logistic and least absolute deviation regression) to evaluate the performance of sign-constrained regression by comparing performance with the cross-validated Lasso under the Toeplitz design. 
Overall, we have confirmed the followings: (1) When the correlation parameter $\rho$ in the Toeplitz design is large, sign-constrained regression performs well. (2) In the case of the least absolute deviation regression with standard Laplace distribution errors, sign-constrained regression typically performs better than the cross-validated Lasso.

\section*{Acknowledgements}

Yuta Koike's research was supported by JST CREST and JSPS Grant-in-Aid for Young Scientists (B) Grant Number JP16K17105.
Yuta Tanoue's research was supported by JSPS Grant-in-Aid for Research Activity start-up Grant Number 17H07322.

{\small
\renewcommand*{\baselinestretch}{1}\selectfont
\addcontentsline{toc}{section}{References}

}

\end{document}